\documentclass[12pt]{article}
\usepackage[centertags]{amsmath}
\usepackage{fancyhdr}
\usepackage{amsfonts}
\usepackage{latexsym}
\usepackage{amssymb}
\usepackage{amsthm}
\usepackage{amsmath,mathrsfs}
\usepackage{amsmath,amscd}
\usepackage{mathtools}
\usepackage[matrix,arrow,curve,cmtip]{xy}
\title{\textbf{Motifs in Derived Algebraic Geometry}}
\author{Renaud Gauthier \footnote{rg.mathematics@gmail.com} \\ \\}
\theoremstyle{definition}

\newtheorem{Thm}{Theorem}[section]
\newtheorem{sAopRKlemma}{Lemma}[subsubsection]
\newtheorem{sAopRKCor}[sAopRKlemma]{Corollary}
\newtheorem{hfpequ}[sAopRKlemma]{Lemma}
\newtheorem{equforcs}[sAopRKlemma]{Lemma}
\newtheorem{ffqczero}[sAopRKlemma]{Lemma}
\newtheorem{cscommutes}[sAopRKlemma]{Lemma}
\newtheorem{ccXh}[sAopRKlemma]{Lemma}
\DeclareMathOperator*{\rrarrop}{\rightrightarrows}

\DeclareMathOperator*{\tensprod}{\otimes}

\newcommand{\beq}{\begin{equation}}
\newcommand{\eeq}{\end{equation}}

\newcommand{\rarr}{\rightarrow}

\newcommand{\larr}{\leftarrow}
\newcommand{\Rarr}{\Rightarrow}

\newcommand{\Ob}{\text{Ob\,}}

\newcommand{\xrarr}{\xrightarrow}

\newcommand{\cA}{\mathcal{A}}

\newcommand{\cC}{\mathcal{C}}

\newcommand{\bL}{\mathbb{L}}

\newcommand{\bR}{\mathbb{R}}

\newcommand{\Cat}{\text{Cat}}

\newcommand{\Hom}{\text{Hom}}
\newcommand{\Ho}{\text{Ho}\,}

\newcommand{\Map}{\text{Map}}

\newcommand{\Mod}{\text{Mod}}

\newcommand{\op}{\text{op}}

\newcommand{\Set}{\text{Set}}
\newcommand{\Spec}{\text{Spec\,}}

\newcommand{\AffC}{\cA \text{ff}_{\cC}}

\newcommand{\AopRD}{A^{\op, \RD}}
\newcommand{\AopRbD}{A^{\op, \mathbb{R} \partial \Delta^n}}

\newcommand{\BopRD}{B_*^{\op, \RD}}
\newcommand{\BopRbD}{B_*^{\op, \mathbb{R} \partial \Delta^n}}

\newcommand{\bfA}{\mathbf{A}}
\newcommand{\bfB}{\mathbf{B}}
\newcommand{\bfuA}{\mathbf{\underline{A}}}
\newcommand{\bfuB}{\mathbf{\underline{B}}}

\newcommand{\Comm}{\text{Comm}}

\newcommand{\CommC}{Comm(\cC)}
\newcommand{\cskCAlg}{\text{csk-CAlg}}
\newcommand{\ccAopRbD}{cc_*(\AopRbD)}
\newcommand{\csBop}{(cs_*(B_*))^{\op}}
\newcommand{\csBopuD}{(cs_*(B_*))^{\op, \uD}}
\newcommand{\csBopD}{(cs_*(B_*))^{\op, \Delta^n}}
\newcommand{\csBopRD}{(cs_*(B_*))^{\op, \RD}}
\newcommand{\csBopRbD}{(cs_*(B_*))^{\op, \mathbb{R} \partial \Delta^n}}
\newcommand{\csAop}{(cs_*A)^{\op}}

\newcommand{\csAopRD}{(cs_*A)^{\op, \RD}}
\newcommand{\csAopRbD}{(cs_*A)^{\op, \mathbb{R} \partial \Delta^n}}

\newcommand{\ccBopRD}{cc_*(\BopRD)}
\newcommand{\ccBopRbD}{cc_*(\BopRbD)}
\newcommand{\ccAopRD}{cc_*(\AopRD)}

\newcommand{\csAopRK}{(cs_*A)^{\op, \RK}}
\newcommand{\csAopRuK}{(cs_*A)^{\op, \RuK}}

\newcommand{\dkAff}{\text{d}k\text{-Aff}}
\newcommand{\dSt}{\text{dSt}}

\newcommand{\Dop}{\Delta^{\op}}

\newcommand{\del}{\partial}

\newcommand{\ettop}{\acute{e}\text{t}}

\newcommand{\Gpd}{\text{Gpd}}

\newcommand{\holim}{\text{holim}}

\newcommand{\kCAlg}{k\text{-CAlg}}
\newcommand{\kMod}{k\text{-Mod}}
\newcommand{\kAff}{\text{k-Aff}}

\newcommand{\RuD}{\mathbb{R} \underline{\Delta}^n}
\newcommand{\RbD}{\mathbb{R} \partial \Delta^n}
\newcommand{\RD}{\mathbb{R} \Delta^n}
\newcommand{\RK}{\mathbb{R}K}
\newcommand{\RuK}{\mathbb{R} \underline{K}}

\newcommand{\Sh}{\text{Sh}}

\newcommand{\SetD}{\Set_{\Delta}}

\newcommand{\skMod}{\text{s}k\text{-Mod}}
\newcommand{\skCAlg}{\text{s}k\text{-CAlg}}
\newcommand{\skCAlgop}{\skCAlg^{\op}}

\newcommand{\skAff}{\text{s}k\text{-Aff}}

\newcommand{\sdkAff}{\text{sd}k\text{-Aff}}

\newcommand{\sAopRubD}{(sA)^{\op, \mathbb{R} \ubD}}
\newcommand{\SpsB}{\Spec s(B_*)}
\newcommand{\SpsA}{\Spec s(A)}
\newcommand{\sBop}{(s(B_*))^{\op}}
\newcommand{\sBopuD}{(s(B_*))^{\op, \uD}}
\newcommand{\sBopD}{(s(B_*))^{\op, \Delta^n}}
\newcommand{\sBopRD}{(s(B_*))^{\op, \RD}}
\newcommand{\sBopRbD}{(s(B_*))^{\op, \mathbb{R} \partial \Delta^n}}
\newcommand{\sAop}{(sA)^{\op}}

\newcommand{\sAopRuD}{(sA)^{\op, \RuD}}

\newcommand{\sAopRD}{(sA)^{\op, \RD}}
\newcommand{\sAopRbD}{(sA)^{\op, \mathbb{R} \partial \Delta^n}}
\newcommand{\sAopRK}{(sA)^{\op, \RK}}
\newcommand{\sAopRuK}{(sA)^{\op, \RuK}}

\newcommand{\uA}{\underline{A}}
\newcommand{\uD}{\underline{\Delta}^n}
\newcommand{\ubD}{\underline{\partial \Delta}^n}
\newcommand{\uK}{\underline{K}}

\begin{document}
\maketitle
\begin{abstract}
We formalize the concept of sheaves of sets on a model site by considering variables thereof, or motifs, and we construct functorially defined derived algebraic stacks from them, thereby eliminating the necessity to choose derived extensions as explained in \cite{TV6}.
\end{abstract}

\newpage

\section{Introduction}
It has been observed in the past (\cite{D1}, \cite{D2}, \cite{TV2}) that one can formalize Algebraic Geometry by writing it in purely categorical terms by simply starting with some symmetric monoidal base category $\cC$ and by considering its category $\CommC$ of commutative and unital monoids, and letting $\AffC = (\CommC)^{\op}$ be the category of affine schemes over $\cC$. On $\AffC$ one then puts some topology $\tau$, be it the Zariski, \'etale, ffqc or any topology that one so wishes to then develop a notion of stacks and higher stacks on $(\AffC, \tau)$. One would obtain in this manner what is referred to as Relative Algebraic Geometry, classical Algebraic Geometry corresponding to the case of having $\cC= (\mathbb{Z}-\Mod, \otimes)$. It is in an attempt to develop Relative Algebraic Geometry over symmetric monoidal $\infty$-categories that Toen observed that one could use model categories since they give rise to $\infty$-categories via the Dwyer-Kan simplicial localization technique. Thus one starts with a symmetric monoidal model category $(\cC, \otimes)$ in the sense of \cite{Ho}, thereby introducing Homotopical Algebraic Geometry (\cite{T2},\cite{TV2}, \cite{TV}, \cite{TV4}), or Algebraic Geometry over model categories. In particular if for some fixed commutative ring $k$ one considers $\cC = \skMod$, the category of simplicial objects in $\kMod$, one obtains Derived Algebraic Geometry (\cite{TV4}).\\

The need to introduce stacks, higher stacks and derived stacks can be seen from a classification problem perspective. As recounted in \cite{T}, one can start with a contravariant functor F from a category $\cC$ of geometric objects to $\Set$, where for $X$ a geometric object, $F(X)$ classifies families of objects parametrized by $X$, and $F$ being valued in $\Set$ this classification is really done up to equality. One may relax that condition and ask that classification be done up to isomorphism as well, whence the introduction of contravariant functors into $\Gpd$, which corresponds to considering 1-stacks. Another motivation for making such a generalization as pointed out in \cite{TV6} is that such $\Set$-valued moduli functors $F$ may not be representable and only admit a coarse moduli space, and not of the expected dimension, so following \cite{Ko}, a natural approach amounts to seeing such spaces as truncations of higher spaces, or derived spaces, smooth (as opposed to being singular), and of the expected dimension. This corresponds to seeking natural extensions of $F$ to functors $F_1:\cC^{\op} \rarr \Gpd$ that make the following diagram commutative, where we take $\cC = \kCAlg = \Comm(\kMod)$, $k$ a commutative ring, since that will be our main point of interest:
\beq
\xymatrix{
	\kCAlg \ar[dr]_{F_1} \ar[r]^F & \Set \\
	& \Gpd \ar[u]_{\pi_0}
} \nonumber
\eeq
Next one may want to further relax the classification scheme by also allowing classification up to equivalence. At an elementary level this means having a $\Cat$-valued functor, but morphisms in $\Cat$ would have to be inverted, and this is possible only if we have functors valued in $\infty$-categories, hence we consider $\infty$-stacks, or stacks for short. This corresponds to looking at extensions $F_{\infty}$ of $F_1$ and $F$ that make the following diagram commutative:
\beq
\xymatrix{
	\kCAlg \ar[r]^-F \ar[dr]^{F_1} \ar[ddr]_{F_{\infty}} & \Set \\
& \Gpd \ar[u]_{\pi_0} \\
& \SetD \ar[u]_{\Pi_1}
} \nonumber
\eeq
Finally, if one considers obstruction theory, as pointed out in \cite{T} Derived Algebraic Geometry is a natural formalism in which such a theory can be written out, and this corresponds to finding derived extensions $\mathbb{R}F$ of $F$, $F_1$ and $F_{\infty}$ that make the following diagram commutative:

\beq
\xymatrix{
	\kCAlg \ar[dd]_i \ar[r]^-F \ar[dr]^{F_1} \ar[ddr]_{F_{\infty}} & \Set \\
& \Gpd \ar[u]_{\pi_0} \\
\skCAlg \ar[r]_-{\bR F} & \SetD \ar[u]_{\Pi_1}
} \label{ToenCD}
\eeq

Now as discussed in \cite{TV6}, not all such extensions will work, there are constraints that have to be met in addition to having the above diagram commutative, for instance having the right derived tangent stack, something that would be known at the onset. Moreover there is no canonical choice of a derived extension. This is our main motivation for introducing a formalism where one would not have to worry about having to pick a derived extension. In addition we would like to functorially construct a derived extension, as exposed in the section that follows.

\section{Construction}
We reproduce \eqref{ToenCD} above as it is presented in \cite{T}:

\beq
\xymatrix{
	\kCAlg \ar[dd]_i \ar[r]^-{sheaves} \ar[dr]^{1-St} \ar[ddr]_{St} & \Set \\
& \Gpd \ar[u]_{\pi_0} \\
\skCAlg \ar[r]_-{\substack{ der.\\ stacks}} & \SetD \ar[u]_{\Pi_1}
} \nonumber
\eeq

where for $X \in \SetD$, $\Pi_1(X) = G \gamma(X)$, $\gamma(X)$ the path category of $X$, and $G: \Cat \rarr \Gpd$, $A \mapsto GA = A[\Sigma^{-1}]$, $\Sigma = A_1$ (\cite{JT}, \cite{GoJa}). We have $\Ob(\Pi_1(X)) = X_0$, morphisms in $\Pi_1(X) = A_1$ and formal inverses. We also have $\pi_0: \Cat \rarr \Set$ is the connected component functor on categories (\cite{McL}).\\

This diagram is made commutative by selecting derived extensions to $\skCAlg$ of the concepts of sheaves, 1-stacks and stacks, which depends very much on the context, so one may inquire whether working with motifs, or variables, instead of specializations thereof, would provide something that is less of an ad hoc construction. This is motivated in particular by the fact that mentioning ``sheaves'' and ``derived stacks'' in the above diagram can be made precise using the formalism of motifs as defined presently. We briefly remind the reader of the definition of motivic frame as introduced in \cite{RG1}: given a mathematical construct $X$, one can formalize its construction by using what is called a motivic frame $x = \{x^{(n)}, x^{(n)} \xrarr{z^{(n)}} x^{(n+1)} \}_{n \geq 0}$ with $z^{(n)}$ gluing maps and $x^{(n)}$ variables, or motifs, the idea being that $X$ is an object of a specific category, whose objects are of a certain (higher) type. By this we mean a type of object may be built from objects of a lower type, in an inductive fashion. In a first approximation, those lower type objects are collected into classes, and the $x^{(i)}$'s would provide variables, or motifs, for objects of each such class. The gluing maps indicate in what manner are the different motifs $x^{(i)}$ put together to form the motif type for $X$. The choice of the word motif is mainly one of semantics; ``variable" is a rather crude way to refer to the $x^{(i)}$'s, relative to ``motif", which is essentially a pattern, originally meant to describe a unifying, elementary model, which captures the nature of the objects within each class in addition to being a simple variable. A motif depending on other motifs will simply be referred to as a higher motif.\\

We let $\Sh$ be a motif for sheaves of sets on $\kCAlg$. Strictly speaking, we have an overarching motif $\Sh$ on a site motif $(M, \tau)$. If $M = \kCAlg$, then we specialize $\Sh$ to that particular site, so it should be denoted $\Sh[\kCAlg]$, but for notation's sake, we will just denote it by $\Sh$ again, since those are the only sheaves we consider in this work. In this situation, $\Sh$ itself is a motif, for sheaves on $\kCAlg$. Hence we are looking for a commutative diagram of motifs. Further we regard the move from sheaves to derived stacks as a functorial operation, thus we would like to replace $\pi_0 \circ \Pi_1$ on the right by an adjoint map $\Set \rarr \SetD$ denoted $j$ that we would like to argue is formally identical to $i$ on the left. This is made possible by the following observation: the connected component functor can also be defined from $\SetD$ and this is what we will use, we define it as the coequalizer (\cite{JT}):
\beq
\Pi_1(X)=X_1 \rrarrop_{d_0}^{d_1} X_0 \rarr \pi_0(X) \nonumber
\eeq
for $X \in \SetD$. Thus functorially on $\SetD$, the connected components functor corresponds to $\pi_0 \circ \Pi_1$. This gives us $\pi_0$ as $\pi_0 \dashv cs_*$, $cs_*$ the constant simplicial functor:
\begin{align}
cs_*: \Set &\rarr \SetD \nonumber \\
S & \mapsto cs_*(S) \nonumber
\end{align}
where $cs_*(S)_n = S$ for all $n \geq 0$ and all face and degeneracy maps are $id_S$. By definition, $i = cs_*$, where $i:\kCAlg \rarr \skCAlg$ is the natural inclusion functor that sends a $k$-algebra $A$ to the constant simplicial object $A$ in $\skCAlg$ (\cite{T}). Let $s: A \rarr sA$ be the functor that maps a structured object $A$ to its simplicial version, a motif for simplicialization. We have $s|_{\kCAlg}$ has $i$ for specialization, and $s|_{\Set}$ has $j$ for specialization, and both correspond to the functor $ cs_*$.\\

Our aim is to show that a motif $\dSt$, depending on the motifs $s$ and $\Sh$, hence a higher motif, making the following diagram commute, is a motif for derived stacks:
\beq
\xymatrix{
	\kCAlg \ar[d]_s \ar[r]^{\Sh} & \Set \ar[d]^s \\
	\skCAlg \ar[r]_{\dSt} & \SetD
} \label{CD}
\eeq
The interest of having such a motif is that it would comprise all possible derived extensions given $s$ and $\Sh$. The aim of the present work is to prove that such a higher motif stands for a derived stack variable, i.e. that one can functorially construct derived stacks, and that such a construction is choice-free if written in the language of motifs. \\

Note that if one uses $s|_{\kCAlg}$ on the left and $s|_{\Set}$ on the right in the above diagram, both being motifs in their own right, in the language of motivic frames, one has a commutative diagram of motifs of degree 1. Observing that both motifs are really the same, namely simplicializing motifs, if one denotes by $s$ their overarching motif, using $s$ instead in the same commutative diagram elevates it to a diagram of motifs of degree 2. Here we use the convention that a motif of degree $i$ can be seen as a motif of degree $j$ for $j \geq i$, taking all gluing maps beyond the $i$-th level to be trivial. Thus $\Sh$ and $\dSt$ though they are of degree 1, are regarded as degree 2 motifs so that our resulting diagram is indeed a commutative diagram of motifs of degree 2. For notation's sake however, we will just use $s$ for either motif $s|_{\kCAlg}$ or $s|_{\Set}$. In other terms we implicitly work with specializations for the sake of building up the desired commutative diagram. \\

We ask that $s$ be a motif for a functorial simplicial frame, which in addition preserves finite limits. The important point being made here is that though one could just work with $cs_*$, one would like a functor that is equivalent to it, but is otherwise unspecified. Thus we work in greater generality. Functorial simplicial frames allow us to do that. To fix ideas, recall from \cite{Hi} that for $A$ an object of $\kCAlg$, we define a simplicial frame on $A$ to be a simplicial object $\hat{A} \in (\kCAlg)^{\Dop} = \skCAlg$, together with an equivalence $cs_*A \rarr \hat{A}$ in the Reedy model category structure on $\skCAlg$ such that the induced map $A \rarr \hat{A}_0$ is an isomorphism and if $A$ is fibrant in $\kCAlg$, so is $\hat{A}$ in $\skCAlg$. We then define a functorial simplicial frame on $\kCAlg$ to be given by a pair $(G,j)$, where $G: \kCAlg \rarr \skCAlg$ is a functor, $j:id \Rarr G$ is such that $j_A: cs_*A \rarr G(A)$ is a simplicial frame for any $A \in \kCAlg$. We abuse notation and refer to $G$ from $ (G,j)$ as the functorial simplicial frame. Then $s$ stands for a motif of functorial simplicial frame, $G$ being a specialization thereof.\\

Given a sheaf motif $\Sh$, and a functorial simplicial frame motif $s$ on $\kCAlg$, we have a corresponding higher motif $\dSt[\Sh,s] = \dSt$ which we aim to prove is a motif for derived stacks, so must satisfy the properties of stacks. Recall from \cite{T} that in the definition of a derived stack we have to use hypercovers, which necessitates the introduction of coaugmented, cosimplicial objects $\mathbf{A} \rarr \mathbf{B}_*$ in $\skCAlg$. Here we take $\mathbf{A}= s(A)$ for $A \in \kCAlg$, and $\mathbf{B}_*$ is a functor:
\begin{align}
\mathbf{B}_*: \Delta &\rarr \skCAlg \nonumber \\
n &\mapsto \mathbf{B}_n = s(B_n) \nonumber
\end{align}
This is a fundamental assumption in our work. Since we generate derived stacks from sheaves by using the motif $s$ for simplicialization, we regard simplicial commutative $k$-algebras $\mathbf{A}$ to come from some $A \in \kCAlg$, but we also regard the $\mathbf{B}_n$ above as being images of $B_n \in \kCAlg$. To fix notations, $\mathbf{B}_* \in \text{cs}( \skCAlg) $ where $\text{cs}$ stands for cosimplicial. In this work, we will consider coaugmented cosimplicial objects of the form $sA \rarr sB_*$, for $A \in \kCAlg$, $B_* \in \cskCAlg$. To see that this still produces a cosimplicial object, note that $s$ is defined on $\kCAlg$. Thus $sB_n$ makes sense for all $[n] \in \Delta$. Now functorially this makes $sB_*$ into a cosimplicial object of $\skCAlg$, since the functor $s: \kCAlg \rarr \skCAlg$ preserves the cosimplicial identities.\\

The reason for introducing cosimplicial objects is elementary; we will argue the opposite of an algebra $A \in \skCAlg$ is an object $\Spec A \in \dkAff$, and it is over such objects that we have hypercovers, given by simplicial objects $\Spec B_* \rarr \Spec A$. To see that taking a cosimplicial co-augmentation produces just that, it suffices to start from $B_*: \Delta \rarr \skCAlg$, so that $B_*^{\op} = \Spec B_*: \Delta^{\op} \rarr \skCAlgop = \dkAff$ is indeed in $\sdkAff$.\\

We are now ready to give the conditions $\dSt$ have to satisfy to be a derived stack. According to \cite{T} applied to our setting, the following conditions must be met:
\begin{itemize}
\item For any equivalence $s(A) \rarr s(B)$ in $\skCAlg$, $A,B \in \kCAlg$, the induced morphism $\dSt(sA) \rarr \dSt(sB)$ is an equivalence in $\SetD$.\\
\item For any coaugmented, cosimplicial object $s(A) \rarr s(B_*)$, corresponding to a \'et-hypercovering in $\dkAff = \skCAlg^{\op}$, the induced morphism $\dSt(sA) \rarr \holim_{n \in \Delta} \dSt(sB_n)$ is an equivalence in $\SetD$.\\
\item For any finite family $\{sA_i\}$ in $\skCAlg$, the natural morphism $\dSt(\prod sA_i) \rarr \prod \dSt(sA_i)$ is an equivalence in $\SetD$.\\
\end{itemize}

Since the construction of derived stacks is intimately linked to that of sheaves of sets on $\kCAlg$, we also give their definition as given in \cite{TV4}. We regard sets as constant simplicial sets. A functor $F: \kCAlg \rarr \Set$ is a sheaf if:
\begin{itemize}
\item For any equivalence $A \rarr B$ in $\kCAlg$, the induced morphism $F(A) \rarr F(B)$ is an equivalence in $\SetD$.\\
\item For any finite family $\{A_i\}_{i \in I}$ in $\kCAlg$, the natural morphism
\beq
F(\prod A_i) \rarr \prod_{i \in I} F(A_i) \nonumber
\eeq
is an equivalence in $\SetD$.\\
\item For any cosimplicial object $A \rarr B_*$ in $\kCAlg$ corresponding to a \'et-hypercover $\Spec B_* \rarr \Spec A$ in $\kAff$, the induced morphism $F(A) \rarr \holim_{n \in \Delta} F(B_n)$ is an equivalence in $\SetD$.
\end{itemize}

\section{Statement of the theorem and proof}
\begin{Thm} \label{Thm}
A higher motif $\dSt$ as defined by the commutative diagram \eqref{CD} is a derived stack.
\end{Thm}
There are three points to be checked, and each will be the subject of a subsection.

\subsection{$\dSt$ preserves equivalences}
Suppose $s(A) \rarr s(B)$ is an equivalence in $\skCAlg$, for $A,B \in \kCAlg$. If we denote equivalences by $\sim$, owing to the fact that $s$ is a functorial simplicial frame, we have:
\beq
\xymatrix{
	cs_*(A) \ar[d] \ar[r]^{\sim} & s(A) \ar[d]^{\sim} \\
cs_*(B) \ar[r]_{\sim} &s(B)
} \nonumber
\eeq
$cs_*(A) \rarr cs_*(B)$ is an equivalence by using the 2-3 property twice in the above diagram, so componentwise this gives an equivalence $A \rarr B$ in $\kCAlg$. Now $\dSt$ is built from a motif of sheaves $\Sh$ which is a sheaf, and by the sheaf condition $A \xrarr{\sim} B$ implies $\Sh(A) \xrarr{\sim} \Sh(B)$ in $\Set$. Since each of $\Sh(A)$ and $\Sh(B)$ are sets and are regarded as constant simplicial sets in $\SetD$, this reads $cs_* \Sh(A) \xrarr{\sim} cs_*\Sh(B)$. Since $s$ is a functorial simplicial frame motif, we have a commutative diagram:
\beq
\xymatrix{
	cs_* \Sh(A)\ar[d]_{\sim} \ar[r]^{\sim} & s(\Sh(A)) \ar[d]^p  \\
cs_* \Sh(B) \ar[r]_{\sim} &s(\Sh(B))
} \nonumber
\eeq
and by using the 2 out of 3 property twice in this diagram we have that $p$ is a weak equivalence. Finally by commutativity of \eqref{CD} $\dSt(sA) = s(\Sh(A)) \xrarr{\sim} s(\Sh(B)) = \dSt(sB)$ so $\dSt$ does preserve equivalences.

\subsection{$\dSt$ satisfies hyperdescent}
Consider a coaugmented cosimplicial object $s(A) \rarr s(B_*)$ in $\skCAlg$ corresponding to a \'et-hypercover $\Spec(s(B_*)) \rarr \Spec(s(A))$ in $\dkAff = \skCAlg^{\op}$. Recall what this means: $sA$ is regarded as a constant cosimplicial object of $\skCAlg$. We have to show $\dSt(sA) \xrarr{\sim} \holim_{n \in \Delta} \dSt(sB_n)$ in $\SetD$. Recall from \cite{TV2} that the \'etale topology on $\dkAff$ induces \'etale hypercoverings: if $\Spec s(A)$ is an object of the site $(\dkAff, \ettop)$, a homotopy \'et-hypercover of $\Spec s(A)$ is given by some $\Spec s(B_*)$ in $\Ho(\sdkAff)$ along with a morphism $\Spec s(B_*) \rarr \Spec s(A)$ in $\Ho(\sdkAff)$ such that for all $n\geq 0$:
\beq
\SpsB^{\RD} \rarr \SpsB^{\RbD} \times^h_{\SpsA^{\RbD}} \SpsA^{\RD} \nonumber
\eeq
is a \'et-covering in $\Ho(\dkAff)$. From \cite{T} a finite family $\{f_i: sA \rarr sB_i \}_{i \in I}$ in $\skCAlg$ is a \'et-covering if and only if there is a finite subset $J$ of $I$ such that: 
\beq
\coprod_{j \in J} \Spec \pi_0 (sB_j) \rarr \Spec \pi_0(sA) \nonumber
\eeq
is \'etale and surjective, and for all $i \in I$:
\beq
\pi_*(sA) \otimes_{\pi_0(sA)} \pi_0(sB_i) \rarr \pi_*(sB_i) \nonumber
\eeq
is an isomorphism.\\

We first have to show that given the conditions on $s(A) \rarr s(B_*)$ above, $A \rarr B_*$ is a coaugmented cosimplicial object in $\kCAlg$ corresponding to a \'et-hypercovering $\Spec B_* \rarr \Spec A$ in $\kAff$. The idea is that we are projecting down to the level of sheaves the descent condition, and we eventually use the functoriality of $s$ to conclude that the motif $\dSt$ satisfies hyperdescent as well.\\

\subsubsection{$ \Spec B_* \rarr \Spec A$ is a \'et-hypercovering in $\kAff$}
The aim of this subsection is to show that for all $n \geq 0$:
\beq
\Spec B_*^{\RD} \rarr \Spec B_* ^{\RbD} \times^h _{\Spec A^{\RbD}} \Spec A^{\RD} \label{zCD}
\eeq
is a \'et-covering in $\Ho(\kAff)$. This will be done in two steps. We will first show
\beq
\csBopRD \rarr \csBopRbD \times^h_{\csAopRbD} \csAopRD \nonumber
\eeq
is a \'et-covering in $\Ho(\dkAff)$, from which it will follow \eqref{zCD} is a \'et-covering in $\Ho(\kAff)$.\\

We first rewrite:
\beq
\SpsB^{\RD} \rarr \SpsB^{\RbD} \times^h_{\SpsA^{\RbD}} \SpsA^{\RD} \nonumber
\eeq
as:
\beq
\sBopRD \rarr \sBopRbD \times^h _{\sAopRbD} \sAopRD \nonumber
\eeq
From \cite{TV} $\sdkAff$ being a simplicial model category it is tensored and cotensored over $\SetD$, hence for $F_* \in \sdkAff$, since $\Delta^n \in \SetD$, we have a well-defined product $\uD \otimes F_*$, and with $\sBop \in \sdkAff$, we have by adjunction:
\beq
\Hom(\uD \otimes F_* , \sBop) \cong \Hom(F_*, \sBopuD) \nonumber
\eeq
as well as:
\beq
\Hom(\uD \otimes F_*, \csBop) \cong \Hom(F_*, \csBopuD) \nonumber
\eeq
Since the exponential map is natural in both arguments, having a map $\sBop \rarr \csBop$, we have maps:
\beq
\xymatrix{
	\sBopuD \ar[d]_{\partial_0} \ar[r]^{\phi} & \csBopuD \ar[d]^{\partial_0} \\
\sBopD \ar[d]_{\gamma} \ar[r]^{\phi_0} & \csBopD  \ar[d]^{\gamma}\\
\sBopRD \ar[r]_{\gamma \phi_0} &\csBopRD
} \label{gammaphi0}
\eeq
where $\partial_0$ is the degree zero map with $\phi_0$ the map on degree zero elements induced by $\phi$, and $\gamma$ is the canonical functor from a given category to its homotopy category. In the same manner we have maps:
\beq
\sBopRbD \rarr \csBopRbD \label{gammaphi1}
\eeq
\beq
\sAopRD \rarr \csAopRD \label{gammaphi2}
\eeq
and:
\beq
\sAopRbD \rarr \csAopRbD \label{gammaphi3}
\eeq
which induce a map of homotopy fiber products:
\beq
\xymatrix{
	\sBopRbD \times^h_{\sAopRbD} \sAopRD \ar[d]\\
\csBopRbD \times^h_{\csAopRbD} \csAopRD
} \label{hfp}
\eeq
which when combined with the bottom horizontal map of \eqref{gammaphi0} yields:
\beq
\xymatrix{
	\sBopRD \ar[d] \ar[r] &\sBopRbD \times^h_{\sAopRbD} \sAopRD \ar[d] \\
\csBopRD  &\csBopRbD \times^h_{\csAopRbD} \csAopRD
}
\eeq
The top horizontal map is a \'et-covering in $\Ho(\dkAff)$. We would like to fill the bottom map and show that it is a \'et-covering as well. In order to do this we will show that both vertical maps are equivalences, which will give us a bottom horizontal map since we work in $\Ho(\dkAff)$. This will also tell us this map gives us a \'et-covering. In a first time to show the vertical maps above are weak equivalences, we will need all maps in \eqref{gammaphi0}, \eqref{gammaphi1}, \eqref{gammaphi2} and \eqref{gammaphi3} to be equivalences so we prove the more general fact:
\begin{sAopRKlemma}
For $K= \Delta^n$ or $K= \partial \Delta^n$, $A \in \cskCAlg$, the map $\sAopRK \rarr \csAopRK$ is a weak equivalence.
\end{sAopRKlemma}
\begin{proof}
We use the fact that $X_*^{\RK} = (X_* ^{\RuK})_0$ as shown in \cite{TV4}, obtained by first taking a fibrant replacement of $X_*$, followed by taking the exponential by $\uK$, and then taking the degree zero component. We will use the following fact from \cite{Hi}: that $\sdkAff$ being a simplicial model category, for any fibrant objects $X$ and $Y$ in $\sdkAff$, $g:X \rarr Y$ is an equivalence if and only if for any cofibrant $Z$ in $\sdkAff$ we have $\Hom(Z, X) \simeq \Hom(Z,Y)$ in $\SetD$. Thus we fix some cofibrant object $F_*$ in $\sdkAff$ and consider
\begin{align}
\Hom(F_*, \sAopRuK) &= \Hom(QF_*, \sAopRuK) \nonumber \\
&\simeq\Hom(\uK \otimes QF_*, R(\sAop)) \nonumber \\
&= \Hom(Q(\uK \otimes F_*), R(\sAop)) \label{square}
\end{align}
Now again we can invoke that same result from \cite{Hi} for $R(\sAop) \rarr R(\csAop)$ since in the diagram below, vertical maps are trivial cofibrations, the top horizontal map is a weak equivalence by definition of $s$, so by the 2 out of 3 property applied twice, the bottom horizontal map is a weak equivalence as well:
\beq
\begin{CD}
\sAop @>>> \csAop \\
@VVV @VVV \\
R(\sAop) @>>> R(\csAop)
\end{CD} \nonumber
\eeq
Being an equivalence from \eqref{square} we have:
\begin{align}
\Hom(F_*, \sAopRuK) &\simeq \Hom(Q(\uK \otimes F_*), R(\sAop)) \nonumber\\
&\simeq \Hom(Q(\uK \otimes F_*), R(\csAop)) \nonumber\\
&\simeq \Hom(F_*, \csAopRuK) \nonumber
\end{align}
hence $\sAopRuK \rarr \csAopRuK$ is a weak equivalence, more precisely, a Reedy equivalence in the Reedy model structure for $\sdkAff$. Those are levelwise equivalences, so its degree zero component $\sAopRK \rarr \csAopRK$ is an equivalence as well.
\end{proof}
\begin{sAopRKCor}
The maps \eqref{gammaphi0}, \eqref{gammaphi1}, \eqref{gammaphi2}, \eqref{gammaphi3} are all weak equivalences.
\end{sAopRKCor}
\begin{proof}
For the first two equations it's immediate, we just take $A = B_*$ in the previous lemma. For the other two equations, $A \in \kCAlg$ gives $sA \in \skCAlg$, hence $\sAop \in \dkAff$, regarded as a constant simplicial object $cs_*(\sAop) \in \sdkAff$. Since $\sAop \xrarr{\sim} \csAop$, we also have $cs_*(\sAop) \xrarr{\sim} cs_*(\csAop)$ since equivalences in the Reedy model structure of $\sdkAff$ are levelwise equivalences. Now for $K = \Delta^n$ or $K = \partial \Delta^n$, following the same reasoning as in the proof of the previous lemma we find that $(cs_*(\sAop))^{\RK} \rarr (cs_*(\csAop))^{\RK}$ is a weak equivalence, and in a simplified notation this reads $\sAopRK \rarr \csAopRK$ is an equivalence. 
\end{proof}
As a consequence of having those equivalences we prove the fiber products in \eqref{hfp} are equivalent:
\begin{hfpequ}
The map of homotopy fiber products:
\beq
	\xymatrix{
	\sBopRbD \times^h_{\sAopRbD} \sAopRD \ar[d]\\
\csBopRbD \times^h_{\csAopRbD} \csAopRD
}
\eeq
is a weak equivalence
\end{hfpequ}
\begin{proof}
We are looking at the following diagram:
\beq
\setlength{\unitlength}{0.5cm}
\begin{picture}(19,21)(0,0)
\thicklines
\put(0,3){$\csBopRbD$}
\put(9,0){$\csAopRbD$}
\put(15,3){$\csAopRD$}
\put(3,7){$\csBopRbD \times^h_{\csAopRbD} \csAopRbD$}
\put(9,13){$\sAopRbD$}
\put(15,16){$\sAopRD$}
\put(0,16){$\sBopRbD$}
\put(3,20){$\sBopRbD \times^h_{\sAopRbD} \sAopRD$}
\put(4,2){\vector(4,-1){4}}
\put(17,2){\vector(-2,-1){2}}
\put(15,6){\vector(2,-1){2}}
\put(5,6){\vector(-2,-1){2}}
\put(1,15){\vector(0,-1){10}}
\put(2,11){$\wr$}
\put(11,12){\vector(0,-1){10}}
\put(12,10){$\wr$}
\put(18,15){\vector(0,-1){10}}
\put(19,12){$\wr$}
\put(17,15.5){\vector(-2,-1){3}}
\put(4,15){\vector(4,-1){4}}
\put(15,19){\vector(2,-1){2}}
\put(5,19){\vector(-3,-1){3}}
\multiput(7,19)(0,-0.3){33}{\line(0,-1){0.15}}
\put(7,9){\vector(0,-1){0.5}}
\end{picture} \nonumber
\eeq
to prove that the vertical map in the back is an equivalence we will use 15.10.10 from \cite{Hi}: in a model category in which we have a diagram such as the one below:
\beq
\setlength{\unitlength}{0.5cm}
\begin{picture}(9,11)(0,0)
\thicklines
\put(0,9){$A$}
\put(5.2,9.2){$B$}
\put(3,6){$A'$}
\put(8,6){$B'$}
\put(0,4){$C$}
\put(5,4){$D$}
\put(3,1){$C'$}
\put(8,1){$D'$}
\put(1,9.5){\vector(1,0){4}}
\put(6,9){\vector(1,-1){2}}
\put(7.5,8){$r_B$}
\put(4,6.5){\vector(1,0){4}}
\put(1,9){\vector(1,-1){2}}
\put(1,7){$r_A$}
\put(5.5,9){\line(0,-1){2}}
\put(5.5,6){\vector(0,-1){1}}
\put(6,7){$p$}
\put(8.5,5.5){\vector(0,-1){3.5}}
\put(9,4){$p'$}
\put(3.5,5.5){\vector(0,-1){3.5}}
\put(0.5,8.5){\vector(0,-1){3.5}}
\put(6,4){\vector(1,-1){2}}
\put(6,2.5){$r_D$}
\put(4,1.5){\vector(1,0){4}}
\put(1,4){\vector(1,-1){2}}
\put(1,2){$r_C$}
\put(1,4.5){\line(1,0){2}}
\put(4,4.5){\vector(1,0){1}}
\end{picture} \nonumber
\eeq
in which all objects are fibrant, the squares in the front and in the back are pullbacks, $p$ and $p'$ are fibrations, $r_B$, $r_C$ and $r_D$ are equivalences, then so is $r_A$. We apply this to our setting, where our initial model category is $\sdkAff$, in which we took fibrant replacements of $sA$, $sB_*$, $cs_*A$ and $cs_*B_*$ already. By 9.3.9 of \cite{Hi} since $\sdkAff$ is a simplicial model category the exponentials of such objects are fibrant as well, and so are their zero components by 15.3.11 of \cite{Hi}, so all objects are fibrant as needed. The pullback squares are given by the homotopy fiber products of the previous diagram, the maps $r_B$, $r_C$ and $r_D$ are the vertical equivalences in that diagram, so $r_A$ corresponds to the map of homotopy fiber products. We take for $p$ and $p'$ the following maps:
\beq
\sAopRD \rarr \sAopRbD \nonumber
\eeq
and:
\beq
\csAopRD \rarr \csAopRbD \nonumber
\eeq
We will use the following fact 9.3.9 2b) from \cite{Hi}: if $X$ is fibrant in a simplicial model category and $j:K \rarr L$ is an inclusion in $\SetD$, then $X^L \rarr X^K$ is a fibration. Applying this to the first map above for example, one gets that:
\beq
\xymatrix{
	(R \sAop)^{\uD} \ar@{=}[d] \ar[r]  &  (R \sAop)^{\ubD} \ar@{=}[d]\\
\sAopRuD & \sAopRubD
} \nonumber
\eeq
is a fibration, more precisely, a Reedy fibration, hence by 15.3.11 of \cite{Hi} the degree zero component is a fibration as well:
\beq
\xymatrix{
	[\sAopRuD]_0 \ar@{=}[d] \ar[r] &[\sAopRubD]_0 \ar@{=}[d]\\
\sAopRD & \sAopRbD
} \nonumber
\eeq
One shows in the same manner that $\csAopRD \rarr \csAopRbD$ is a fibration. This completes the proof that the map of homotopy fiber products is a weak equivalence.
\end{proof}
\begin{equforcs} \label{equforcs}
	The map $\csBopRD \rarr \csBopRbD \times^h_{\csAopRbD} \csAopRD$ is a \'et-covering in $\Ho(\dkAff)$.
\end{equforcs}
\begin{proof}
The previous lemma shows that the right vertical map in the diagram:
\beq
	\xymatrix{
		\sBopRD \ar[d] \ar[r] & \sBopRbD \times^h_{\sAopRbD} \sAopRD \ar[d]\\
	\csBopRD \ar@{.>}[r] & \csBopRbD \times^h_{\csAopRbD} \csAopRD
} \nonumber
\eeq
is an equivalence. The vertical map on the left has already been shown to be an equivalence as well. Hence in $\Ho(\dkAff)$ we have a bottom horizontal map as shown in the diagram above. We show this is a \'et-covering. From \cite{TV4} the definition of coverings uses the higher homotopy groups $\pi_i$ for $i \geq 0$. Recall that for $\cC$ a simplicial model category, $A \in \cC$, $|A| = \Map_{\cC}(1, A) \in \Ho(\SetD)$. If $\cC$ is pointed, $|A|$ has a natural basepoint and we can define $\pi_i(A) = \pi_i(|A|,*)$. Here $\Map_{\cC}$ is the simplicial hom in the simplicial model category $\cC$ which for us we take to be $\skCAlg$. By 9.3.3 of \cite{Hi} since 1 is cofibrant, if $A \xrarr{\sim} B$ is an equivalence of fibrant objects, then $\Map_{\cC}(1,A) \rarr \Map_{\cC}(1,B)$ is an equivalence in $\SetD$, i.e. $|A| \xrarr{\sim} |B|$, and by definition of higher homotopy groups on $\skCAlg$ we conclude that being a \'et-covering is an invariant on weak equivalence classes, whence the result.
\end{proof}
We finally prove:
\begin{ffqczero} \label{ffqczero}
	The map $\BopRD \rarr \BopRbD \times^h_{\AopRbD} \AopRD$ gives a \'et-covering in $\Ho(\kAff)$.
\end{ffqczero}
We start from:
\beq
\xymatrix{
	\csBopRD \ar[d] \ar[r] & \csBopRbD \times^h_{\csAopRbD} \csAopRD \ar[d] \\
\BopRD & \BopRbD \times^h_{\AopRbD} \AopRD
} \nonumber
\eeq
we will show there is a bottom horizontal map making this diagram commutative, which in addition provides a \'et-covering in $\Ho(\kAff)$. First we show that $(cs_*(B_*))^{\op, \bR K} \cong cc_*((B_*)^{\op, \bR K})$ in $\Ho(\sdkAff)$, for $K = \Delta^n$ or $K = \del \Delta^n$, where $cc_*$ denotes the constant cosimplicial functor. Note that this isomorphism still holds if we considered $A$ instead of $B_*$. This allows us to rewrite the top horizontal map above using the constant cosimplicial functor $cc_*$. Then we will argue that:
\beq
\ccBopRbD \times^h _{\ccAopRbD} \ccAopRD \cong cc_* \big( \BopRbD \times^h_{\AopRbD} \AopRD \big) \nonumber
\eeq
and finally using the constant nature of the constant cosimplicial functor $cc_*$ we will show we have a \'et-covering in $\Ho(\kAff)$ as desired. First:
\begin{cscommutes}
	$(cs_*(B_*))^{\op, \bR K} \cong cc_*(B_*^{\op, \bR K})$ in $\Ho(\sdkAff)$ for $K = \Delta^n$ or $K = \del \Delta^n$.
\end{cscommutes}
\begin{proof}
First observe that:
	\begin{align}
		(cs_*(B_*))^{\op, \bR K} &= \{\big(R([cs_*(B_*)]^{\op})\big)^{\uK}\}_0 \nonumber \\
		&= \{\big(R(cc_*(B_*^{\op}))\big)^{\uK}\}_0 \nonumber \\
		&\cong  ((cc_*(RB_*^{\op}))^{\uK})_0 \nonumber 
	\end{align}
	where the isomorphism in $\Ho(\sdkAff)$ is derived as follows: we have a trivial cofibration $B_*^{\op} \rarr RB_*^{\op}$. This equivalence is preserved by the constant cosimplicial functor: $cc_*(B_*^{\op}) \xrarr{\sim} cc_*(RB_*^{\op})$. On the other hand, we also have a trivial cofibration in $\sdkAff$: $cc_*(B_*^{\op}) \xrarr{\sim} Rcc_*(B_*^{\op})$. It follows we have an isomorphism $Rcc_*(B_*^{\op}) \cong cc_*(RB_*^{\op})$ in $\Ho(\sdkAff)$, which gives us the isomorphism above. On the other hand, we have:
	\begin{align}
		cc_*(B_*^{\op, \bR K}) &= cc_*((RB_*^{\op})^K) \nonumber \\
		&= cc_*(((RB_*^{\op})^{\uK})_0) \nonumber
	\end{align}
where the zero-th component is taken in the simplicial model structure of $\skAff$. To emphasize this we can write this last object as $cc_*(((RB_*^{\op})^{\uK})_{0,-})$. Since this zero-th component is taken relative to $B_*$, we can write this as $(cc_*((RB_*^{\op})^{\uK}))_{0,-}$. Thus it suffices to show:
	\beq
	(cc_*(RB_*^{\op}))^{\uK} \cong cc_*((RB_*^{\op})^{\uK}) \nonumber
	\eeq
in $\Ho(\sdkAff)$, where the exponentiation on the left is taken in the simplicial model structure of $\sdkAff$, and the one on the right with respect to the simplicial model structure of $\skAff$. We will show we have an equivalence $(cc_*(RB_*^{\op}))^{\uK} \simeq cc_*((RB_*^{\op})^{\uK})$ in $\sdkAff$. For $F_*$ cofibrant in $\sdkAff$, we have:
\beq
	\Hom_{\sdkAff}(F_*, (cc_*(RB_*^{\op}))^{\uK}) \simeq \Hom(F_{*,*} \tensprod_{\sdkAff} \uK, cc_*(RB_*^{\op})) \nonumber
\eeq
where the first index for $F$ corresponds to the simplicial structure of $\sdkAff$, and the second index corresponds to the cosimplicial structure of $\dkAff = (\skCAlg)^{\op} = cs (\kCAlg)^{\op} = cs \kAff$. Now observe:
	\beq
	F_{*,*} \tensprod_{\sdkAff} \uK = \coprod_{\sigma_n \in K_n} F_{n,*} = \coprod_p \big( \coprod_{\sigma_n} F_{n,p} \big) = \coprod_p F_{*,p} \tensprod_{\skAff} \uK \nonumber
	\eeq
It follows:
	\begin{align}
		\Hom_{\sdkAff}(F_{*,*} \otimes \uK,\, &cc_*(RB_*^{\op})) \nonumber \\
		&= \Hom(\coprod_p F_{*,p} \tensprod_{\skAff} \uK, \coprod_p cc_p(RB_*^{\op})) \nonumber \\
		&= \coprod_p \Hom(F_{*,p} \otimes \uK, RB_*^{\op}) \nonumber \\
		&\simeq \coprod_p \Hom_{\skAff}(F_{*,p}, (RB_*^{\op})^{\uK}) \nonumber \\
		&= \coprod_p \Hom(F_{*,p}, cc_p(RB_*^{\op, \uK})) \nonumber \\
		&= \Hom_{\sdkAff}(F_*, cc_*(RB_*^{\op, \uK})) \nonumber
	\end{align}

\end{proof}
By construction, we formally have the same result if we use $A$ instead of $B_*$. We now show:
\begin{ccXh}
\beq
\ccBopRbD \times^h _{\ccAopRbD} \ccAopRD \cong cc_* \big( \BopRbD \times^h_{\AopRbD} \AopRD \big) \nonumber
\eeq
or in other terms, the functor $cc_*$ preserves homotopy fiber products.
\end{ccXh}
\begin{proof}
	We first focus on the homotopy fiber product on the left above. Following \cite{Hi} it is defined by using a functorial factorization which we will denote by $E$: $A \times^h_C B = E(A) \times_C E(B)$. By definition of the constant cosimplicial functor $cc_*$, a map $cc_*(A) \rarr cc_*(B)$ is given by a map $A \rarr B$ which is the same in all degrees, hence can be denoted by $cc_*(A \rarr B)$. By definition of the functorial factorization, $E(cc_*X) \cong cc_*(EX)$ in $\Ho(cs \cC)$ for $X \in \cC$, and this can be shown in the same fashion that we have shown $cc_*(RB_*^{\op}) \cong Rcc_*(B_*^{\op})$, hence we have a map $cc_*A \larr E(cc_*X) \cong cc_*(EX)$ that can be denoted $cc_*(A \larr E(X))$. It follows:
\begin{align}
& cc_*(\BopRbD) \times^h_{cc_*(\AopRbD)} cc_*(\AopRD)\nonumber \\
& = E(cc_*(\BopRbD)) \times_{cc_*(\AopRbD)} E(cc_*(\AopRD)) \nonumber \\
&=\lim \Big( E(\ccBopRbD) \rarr \ccAopRbD \larr E(\ccAopRD) \Big) \nonumber\\
& \cong \lim \Big( cc_*E(\BopRbD) \rarr \ccAopRbD \larr cc_*E(\AopRD) \Big) \nonumber\\
&=\lim cc_*\big( E(\BopRbD) \rarr \AopRbD \larr E(\AopRD)\big) \nonumber\\
&=cc_*\big( \lim E(\BopRbD) \rarr \AopRbD \larr E(\AopRD)\big) \nonumber
\end{align}
since $cc_*$ preserves limits as a right adjoint, and this equals $cc_*( \BopRbD \times^h _{\AopRbD} \AopRD)$ as claimed.
\end{proof}
We can now complete the proof of lemma \ref{ffqczero}: with the previous results the \'et-covering in $\Ho(\dkAff)$ given by Lemma \ref{equforcs} now reads:
\beq
cc_*(\BopRD) \rarr cc_*(\BopRbD \times^h _{\AopRbD} \AopRD) \label{et1}
\eeq
and again by the constant nature of $cc_*$ this map being a \'et-covering implies:
\beq
\BopRD \rarr \BopRbD \times^h _{\AopRbD} \AopRD \label{et2}
\eeq
is a \'et-covering in $\Ho(\kAff)$, i.e. $A \rarr B_*$ corresponds to a homotopy \'et-hypercovering in $\kAff$.\\

This is not immediate, though fairly clear, so we flesh it out. To say that \eqref{et1} is an \'etale covering means it is of the form $\{ \Spec A_i \rarr \Spec A \}_{i \in I}$, a family for which there is a finite $J \subset I$ such that for all $i$ $\pi_*A \otimes_{\pi_0 A} \pi_0 A_i \rarr \pi_* A_i$ is an isomorphism, and $\coprod_{j \in J} \Spec \pi_0 A_i \rarr \Spec \pi_0 A$ is \'etale and surjective. It turns out here \eqref{et1} is a one element family. To say \eqref{et2} is an \'et-covering in $\Ho(\kAff)$ however means it is of the form $\{ \Spec \uA_i \rarr \Spec \uA \}_{i \in I}$, and we have again a finite subset $J \subset I$ such that $\coprod_{j \in J} \Spec \uA_j \rarr \Spec \uA$ is faithfully flat and \'etale. We now show this holds.\\

We first write \eqref{et1} in the form $\Spec \bfB \rarr \Spec \bfA$, where $\bfA$ and $\bfB$ are algebras. This means 
\begin{align}
	\bfB &= \big( \ccBopRD \big)^{\op} \nonumber \\
	&= cs_*(\BopRD)^{\op} \nonumber
\end{align}
and
\begin{align}
	\bfA &= \big( cc_*(\BopRbD \times^h_{\AopRbD} \AopRD) \big)^{\op} \nonumber \\
	&= cs_*( \BopRbD \times^h_{\AopRbD} \AopRD)^{\op} \nonumber
\end{align}
In a first time $\Spec \bfB \rarr \Spec \bfA$ being an \'et-covering means $\pi_*\bfA \otimes_{ \pi_0 \bfA} \pi_* \bfB \rarr \pi_* \bfB$ is an isomorphism and $\Spec \pi_0 \bfB \rarr \Spec  \pi_0 \bfA$ is \'etale and surjective. This means $\bfA \rarr \bfB$ is strongly \'etale by Definition 1.2.6.1 of \cite{TV4}, which in turns means $\bfA \rarr \bfB$ is \'etale by Thm 2.2.2.6 of \cite{TV4}, which means flat and $\pi_0 \bfA \rarr \pi_0 \bfB$ \'etale by Cor. 2.2.2.11 of \cite{TV4}. Being flat means $\bfB \otimes_{\bfA}^{\bL} - $ is exact. However both $\bfA$ and $\bfB$ are constant simplicial objects, so this means $\bfuB \otimes_{\bfuA}^{\bL} -$ is exact, that is $\bfuA \rarr \bfuB$ is flat, i.e $\Spec \bfuB \rarr \Spec \bfuA$ is flat, where we used the notations:
\begin{align}
	\bfuB &= (\BopRD)^{\op} \nonumber \\
	\bfuA &= (\BopRbD \times^h_{\AopRbD} \AopRD)^{\op} \nonumber
\end{align}
that is $\BopRD \rarr \BopRbD \times^h_{\AopRbD} \AopRD$ is flat. To show that $\Spec \bfuB \rarr \Spec \bfuA$ is \'etale, that is $\BopRD \rarr \BopRbD \times^h_{\AopRbD} \AopRD$ is \'etale, it suffices to observe $\Spec \pi_0 \bfB \rarr \Spec \pi_0 \bfA$ is \'etale. This reads:
\beq
	\Spec \pi_0 cs_*(\BopRD)^{\op} \rarr \Spec \pi_0 cs_*( \BopRbD \times^h_{\AopRbD} \AopRD)^{\op} \nonumber
\eeq
Now $\pi_0 \dashv cs_*$, so this can be rewritten:
\beq
\Spec (\BopRD)^{\op} \rarr \Spec ( \BopRbD \times^h_{\AopRbD} \AopRD)^{\op} \nonumber
\eeq
which simplifies as:
\beq
\BopRD \rarr \BopRbD \times^h_{\AopRbD} \AopRD \nonumber
\eeq
\'etale, and flat so far. For the surjectivity, $\Spec \pi_0 \bfB \rarr \Spec \pi_0 \bfA$ is surjective, and doing the same manipulation as above, this reads $\BopRD \rarr \BopRbD \times^h_{\AopRbD} \AopRD$ is surjective.

\subsubsection{From $\Sh$ being a sheaf to $\dSt$ being a derived stack}
Now that $A \rarr B_*$ corresponds to a \'et-hypercovering in $\kAff$, $\Sh$ being a sheaf we obtain an equivalence in $\SetD$: $\Sh(A) \xrarr{\sim} \holim_{n \in \Delta}\Sh(B_n)$. Now by definition of $s$, we obtain the following commutative diagram:
\beq
\xymatrix{
	cs_* \Sh(A) \ar[d]_{\sim} \ar[r]^-{\sim} & cs_* \holim_{n \in \Delta} \Sh(B_n)  \ar[d]^{\sim}\\
s(\Sh(A)) \ar@{=}[d] \ar[r]_-{\sim} & s( \holim_{n \in \Delta} \Sh(B_n)) \ar@{=}[d] \\
\dSt(sA) & \holim_{n \in \Delta} s(\Sh(B_n)) \ar@{=}[d] \\
&  \holim \; \dSt(sB_n)
}
\eeq
where $\dSt(sA)= s\Sh(A)$ by commutativity of \eqref{CD}, $s (\holim \, \Sh(B_n)) = \holim\, s \Sh(B_n)$ since as shown in \cite{Hi} homotopy limits are equalizers, hence limits, and $s$ preserves finite limits by definition, and finally this last object is equal to $\holim\,  \dSt(sB_n)$ again by commutativity of \eqref{CD}. Further Reedy equivalences being defined levelwise, $\Sh(A) \xrarr{\sim} \holim_{n \in \Delta} \Sh(B_n)$ implies $cs_*\Sh(A) \xrarr{\sim} cs_*\holim_{n \in \Delta} \Sh(B_n)$, hence $s\Sh(A) \rarr s ( \holim\,  \Sh(B_n))$ is an equivalence by the 2 out of 3 property, and it follows that $\dSt(sA) \xrarr{\sim} \holim\,  \dSt(sB_n)$ i.e. $\dSt$ satisfies descent.

\subsection{$\dSt$ preserves finite products}
We consider a finite family $\{sA_i \}$ in $\skCAlg$. We need to show the natural morphism:
\beq
\dSt(\prod sA_i) \rarr \prod \dSt(sA_i) \nonumber
\eeq
is an equivalence in $\SetD$. We have $\dSt(\prod sA_i) = \dSt(s \prod A_i)$ since $s$ preserves finite products. This latter object is equal to $s\Sh(\prod A_i)$ by commutativity of \eqref{CD}. Since $\Sh$ is a sheaf, we have $\Sh(\prod A_i) \xrarr{\sim} \prod \Sh(A_i)$ in $\SetD$. Keeping in mind that we consider sets as constant simplicial sets in $\SetD$, this reads $cs_* \Sh(\prod^h A_i) \xrarr{\sim} cs_* \prod \Sh(A_i)$, so we have the following commutative diagram by virtue of the existence of the natural transformation $j$:
\beq
\xymatrix{
	cs_* \Sh(\prod A_i) \ar[d]_{\sim} \ar[r]^{\sim} & cs_* \prod \Sh(A_i)  \ar[d]^{\sim}\\
s\Sh(\prod A_i) \ar@{=}[d] \ar[r]_{\sim} & s \prod \Sh(A_i) \ar@{=}[d] \\
\dSt(s \prod A_i) \ar@{=}[d]  &\prod s(\Sh(A_i)) \ar@{=}[d]\\
\dSt(\prod sA_i) & \prod \dSt(sA_i)
} \nonumber
\eeq
where we used the 2 out of 3 property to have an equivalence $s(\Sh(\prod A_i)) \rarr s\prod \Sh(A_i)$, hence $\dSt(\prod sA_i) \xrarr{\sim} \prod \dSt(sA_i)$. This completes the proof that $\dSt$ is a derived stack, the claim of Theorem \ref{Thm}.

\end{document}